\documentclass[nosecthm,noseceq]{ws-ijnt}
\usepackage[T1]{fontenc}
\usepackage{lmodern,mathtools,booktabs,needspace}
\usepackage{microtype}
\usepackage[hidelinks]{hyperref}
\hypersetup{pdftitle={Nonlinear twists of Rankin--Selberg coefficients in the eigenvalue aspect},pdfauthor={Yangbo Ye}}
\newcommand{\e}{\mathrm e}
\newcommand{\supp}{\operatorname{supp}}
\newcommand{\norm}[1]{\left\lVert #1\right\rVert}
\newcommand{\M}{\mathcal M}
\newcommand{\B}{\mathcal B}

\title{Nonlinear twists of Rankin--Selberg coefficients\\in the eigenvalue aspect}
\author{Yangbo Ye}
\address{Department of Mathematics, The University of Iowa\\
Iowa City, Iowa 52242, United States\\
\email{yangbo-ye@uiowa.edu}}
\renewcommand{\wsdraftnote}{}
\begin{document}
\markboth{Yangbo Ye}{Nonlinear twists of Rankin--Selberg coefficients}

\maketitle
\begin{abstract}
Let $g$ be a fixed primitive Hecke--Maass cusp form, and let $f$ vary
with spectral parameter $k_f=K$, at a fixed level and with trivial
nebentypus. We study smooth sums of
$\lambda_f(n)\lambda_g(n)e(\alpha n^\beta)$ for fixed real
$\alpha\ne0$. The crucial input is the smooth spectral second moment
of Lau, Liu, and Ye, for which the window length can be taken to be
$L=K^{1/3+a}$ for arbitrarily small fixed $a>0$. We show that, for every
$\kappa>0$, there is $\beta_0>0$ such that the sum is
$O(K^{2/3+\kappa}X^{1/2})$ for $1\le X\le K^2$ and
$0<\beta\le\beta_0$. We give a detailed Baire category argument
which makes the dependence on the smooth weight uniform through
one finite derivative norm. A second argument keeps the phase
explicit, applies Voronoi summation and shifted convolution
estimates, and gives the numerical bound
$O(K^{423/224+(5/7)\beta+\varepsilon})$ at $X=K^2$ for
$0<\beta<5/32$.
\end{abstract}

\section{Introduction}

Throughout, \(N\) is fixed and the nebentypus is trivial.
Both forms are primitive Hecke--Maass cusp forms for \(\Gamma_0(N)\),
and the form \(g\) is fixed. Their positive Fourier coefficients
are therefore real in the Hecke normalization. We normalize the Hecke eigenvalues by
\(\lambda_f(1)=\lambda_g(1)=1\), and put
\begin{equation}\label{eq:1}
\e(x)=e^{2\pi i x},\qquad
 S_X(f,g;\alpha,\beta)
 =\sum_{n\geq1}\lambda_f(n)\lambda_g(n)
          \e(\alpha n^\beta)\phi(n/X).
\end{equation}
The positive-integer convention in \eqref{eq:1} is necessary when \(\beta\) is
not an integer. Constants are allowed to depend on \(N,g\), and on the
fixed smooth functions that occur below.

We ask for cancellation which is uniform as the eigenvalue of $f$
tends to infinity. At the conductor scale $X=K^2$, the estimate
$O(K^{2+\varepsilon})$ is the natural comparison. The purpose of
this paper is to obtain a power saving for nonlinear twists in a
positive range of $\beta$ while keeping $g$ fixed.

There are two arguments. The first starts with the smooth moment
in Lau, Liu, and Ye \cite{LLY}. Its implied constant initially
depends on a fixed smooth weight. We prove that one finite
derivative norm controls this dependence simultaneously for all
admissible spectral and summation parameters. We can then insert
the weight $e(\alpha n^\beta)$ and choose a sufficiently small
positive range of $\beta$. The second argument follows the phase
through Kuznetsov and Voronoi summation. It uses a longer spectral
window but yields an explicit range of $\beta$.

The principal conclusion is the following.
\begin{theorem}\label{thm:main}
For every \(\kappa>0\) there exists
\(\beta_0=\beta_0(N,g,\kappa)>0\) with the following property.
For every fixed \(\alpha\in\mathbb R\), every
\(\phi\in C_c^\infty(-1,1)\), and every \(0<\beta\leq\beta_0\),
\begin{equation}\label{eq:2}
|S_X(f,g;\alpha,\beta)|
 \ll_{N,g,\kappa,\alpha,\phi}
 K^{2/3+\kappa}X^{1/2},
 \qquad K=k_f\ge2,\quad 1\leq X\leq K^2.
\end{equation}
Here $k_f$ is real. The implied constant can be chosen uniformly for
\(0<\beta\leq\beta_0\). The threshold \(\beta_0\) is independent of
\(\alpha\) and \(\phi\).
\end{theorem}

At \(X=K^2\), \eqref{eq:2} becomes
\begin{equation}\label{eq:3}
|S_{K^2}(f,g;\alpha,\beta)|\ll K^{5/3+\kappa}
       =X^{5/6+\kappa/2}.
\end{equation}
Thus it is a power saving over \(X\) whenever \(0<\kappa<1/3\).
Taking \(\kappa=5/24\) gives \(X^{15/16}\) for a positive
range of \(\beta\).

The order of the quantifiers matters. The theorem says that for every
fixed \(\kappa\) one can choose a positive range of \(\beta\).
It does not say that one fixed positive range works for every
\(\kappa>0\). In particular, it does not assert a bound
\(K^{5/3+\varepsilon}\) for every \(\varepsilon>0\) at one fixed
positive value of \(\beta\).

\subsection{Earlier subconvexity bounds and related nonlinear sums}

We first recall the eigenvalue and weight aspect of the subconvexity
problem. The fixed-level analytic conductor of $L(1/2+it,f\otimes g)$
has order $K^4$ when $g,t$ are fixed and $K$ is the spectral parameter
of $f$. For a holomorphic $f$, the same scale is obtained when $K$
denotes its weight. Convexity gives $K^{1+\varepsilon}$.
Let $\theta$ be an admissible exponent toward the Ramanujan and
Selberg conjectures. We use $\theta=7/64$ from Kim and Sarnak
\cite{KS}.

Sarnak \cite{S} obtained a subconvexity bound when the varying
form is holomorphic. Liu and Ye \cite{LY} treated varying Maass
forms. There is a gap in Sections~4.14--4.15 of \cite{LY}; the
stronger exponent $(3+2\theta)/4$ stated there is not justified
by that argument. The valid exponent is $(15+2\theta)/16$.
The corrected calculation is given in \cite[Sections~8--9]{LYTrace}
and is explicitly recorded in \cite[(1.3)]{LLYJNT}.

The shifted convolution estimates of Lau, Liu, and Ye \cite{LLYJNT}
give subconvexity even with the trivial value $\theta=1/2$.
Blomer \cite{B04,B05} obtained further shifted convolution estimates
and a slightly stronger subconvexity exponent. Finally,
Lau, Liu, and Ye \cite{LLY} proved the exponent $2/3$, independent
of $\theta$. The relevant exponents are recorded in Table~\ref{tab:history}.
The numerical entries are obtained by substituting $\theta=7/64$;
they need not be the numbers printed in the earliest versions of
the respective papers.

\begin{table}[ht]
\caption{Exponents $\delta$ in $L(1/2+it,f\otimes g)\ll K^{\delta+\varepsilon}$.}
\label{tab:history}
\centering
\begin{tabular}{lll}
\hline
Source & $\delta$ & At $\theta=7/64$\\
\hline
Sarnak \cite{S} (holomorphic $f$) & $18/(19-2\theta)$ & $576/601$\\
Liu--Ye, corrected \cite{LY,LYTrace} & $(15+2\theta)/16$ & $487/512$\\
Lau--Liu--Ye \cite{LLYJNT} & $1-1/(8+4\theta)$ & $119/135$\\
Blomer \cite{B05} & $(6-2\theta)/(7-4\theta)$ & $37/42$\\
Lau--Liu--Ye \cite{LLY} & $2/3$ & $2/3$\\
\hline
\end{tabular}
\end{table}

For the present application, a pointwise bound for the central
$L$-value is not the input we need. We need the smooth coefficient
moment \eqref{eq:5}, with its allowed window length and its dependence on
the weight. We compare earlier versions of this moment immediately
after Proposition~\ref{prop:input}. The passage from a fixed weight
to a varying oscillatory weight is supplied by Lemma~\ref{lem:baire}.

Nonlinear twists of automorphic coefficients have also been studied
in settings where the form is fixed. Ren and Ye
\cite{RY10,RY13,RY14,RY15a,RY15b,RY16} developed resonance and rapid-decay
formulae for single forms and higher-degree coefficients. Czarnecki \cite{Cz16}
studied resonance sums for Rankin--Selberg products.
For products of two fixed $\mathrm{GL}_2$ forms,
Huang, Sun, and Zhang \cite{HSZ} prove estimates for analytic twists;
their implied constants depend on both forms. Such estimates do
not by themselves give the uniformity in $k_f$ required here.

There are also related results obtained by allowing forms to vary.
Ye \cite{Ye22} considered nonlinear sums in the holomorphic weight
aspect, while Salazar and Ye \cite{SY17} studied spectral second
moments for the coefficients of one Maass form. Liu, Sun, and Ye
\cite{LSY20} studied double moments of Rankin--Selberg $L$-functions.
Gillespie, Samanta, and Ye \cite{GSY23} considered the same product
$\lambda_f(n)\lambda_g(n)$ and nonlinear phase as in \eqref{eq:1}, with both
holomorphic weights varying, and obtained bounds outside a small
exceptional set of pairs. Our estimates instead keep $g$ fixed and
apply to each varying primitive Maass form $f$.
The higher-order weighted stationary-phase formulas of McKee, Sun,
and Ye \cite{MSY17} provide related tools for oscillatory transforms.
Qin and Ye \cite{QY23} study spectral averages of resonance sums
for $\mathrm{GL}(3)\times\mathrm{GL}(2)$. More recently,
Shao and Zhang \cite{SZ27} study averages of
$\mathrm{GL}_2\times\mathrm{GL}_2$ resonance sums.
We use a second absolute moment, from which individual estimates
follow by positivity.

\section{The smooth moment supplied by the 2006 proof}

For \(K\geq2\), let \(\B_N(K,L)\) denote the primitive Hecke--Maass
forms of level \(N\), counted once, with real spectral parameters
satisfying \(|k_j-K|\leq L\). For a smooth function \(W\) supported in
\((1,2)\), define
\begin{equation}\label{eq:4}
T_{j,Y}(W)=
 \sum_{n\geq1}\lambda_{f_j}(n)\lambda_g(n)W(n/Y),
 \qquad
 \M_{K,L,Y}(W)=\sum_{f_j\in\B_N(K,L)}
                     |T_{j,Y}(W)|^2.
\end{equation}

\Needspace{13\baselineskip}
We use the following consequence of \cite{LLY}.
\begin{proposition}[Fixed-weight moment]\label{prop:input}
Fix \(0<d<1/3\) and \(a>0\). For every fixed
\(W\in C_c^\infty((1,2))\),
\begin{equation}\label{eq:5}
\M_{K,L,Y}(W)\ll_{W,N,g,d,a} K^{1+a}LY,
\end{equation}
uniformly for
\begin{equation}\label{eq:6}
K\geq2,\qquad 1\leq Y\leq K^2,\qquad
 K^{1/3+d}\leq L\leq K^{1-d}.
\end{equation}
\end{proposition}

Here \eqref{eq:5} is the dyadic moment established in the proof of the 2006
paper, rather than just its theorem about central \(L\)-values.
Its equations (5.1)--(5.4) formulate the sum with a smooth weight
supported in \([1,2]\); equation (11.8) completes the moment estimate
for holomorphic \(g\), and Section 12 gives the Maass modifications.
The proof uses the support of \(W\) and bounds for its derivatives.
No identity special to the approximate functional equation is used
in that part of the proof. Every fixed \(W\) has fixed derivative
bounds, and complex weights may also be handled by taking real and
imaginary parts.

The small parameter in \cite{LLY} can be chosen sufficiently small
relative to \(d\) and \(a\). This accommodates the margins in \eqref{eq:6}, the
constant multiple of that small parameter in the final exponent,
and any factor \(Y^\varepsilon\), since \(Y\leq K^2\).

There is a normalization issue in using this input.
If an \(L^2\)-normalized primitive form has Fourier expansion
\[
 f(z)=\sqrt y\sum_{n\ne0}\rho_f(n)
               K_{ik_f}(2\pi|n|y)\e(nx),
\]
then
\begin{equation}\label{eq:7}
\rho_f(n)=\rho_f(1)\lambda_f(n),\qquad
 \omega_f=\frac{|\rho_f(1)|^2}{\cosh(\pi k_f)}.
\end{equation}
Kuznetsov naturally gives the moment with the weight \(\omega_f\).
For fixed level,
\begin{equation}\label{eq:8}
\omega_f\asymp_N L(1,\operatorname{Ad}f)^{-1},
 \qquad
 \omega_f^{\pm1}\ll_{N,\varepsilon}(1+k_f)^\varepsilon.
\end{equation}
The leading-coefficient comparison is the one used explicitly in
\cite{LLY}. For the adjoint value, see \cite[p.~163, (0.10), and
Corollary~0.3]{HL}, and \cite[Theorem~2]{I90}. In fact, removal of
the harmonic weight requires only the upper bound for
$L(1,\operatorname{Ad}f)$. On the window \eqref{eq:6}, \(k_f\asymp K\)
for sufficiently large \(K\), so removal of \(\omega_f\) costs
\(K^\varepsilon\). This loss is included in \(K^a\) in (5).
The full Fourier-coefficient moment includes oldforms, and its
primitive contribution can be retained by positivity.

We do not normalize arbitrary oldform vectors by dividing by their
first Fourier coefficient: that coefficient can vanish. Statements
for a primitive form of level dividing \(N\) can instead be made
using its usual unscaled lift, with a fixed level-dependent norm
factor.

\subsection{The length of the spectral window}

We make the comparison with earlier smooth moments explicit.
Suppressing arbitrarily small powers and keeping the smooth weight
fixed, the valid Liu--Ye calculation gives
\[
 \M_{K,L,Y}(W)
 \ll K^\varepsilon\left(KLY+
       \frac{K^5Y^{5/2+\theta}}{L^7}\right)
 =K^{1+\varepsilon}LY
       \left(1+\frac{K^4Y^{3/2+\theta}}{L^8}\right).
\]
See \cite[(9.7)--(9.13)]{LYTrace}, which retains the shift range
$|h|\ll K^{2+\varepsilon}c^2/Y$. When $Y\le K^2$, the second
term in parentheses is bounded if
\[
 L\ge K^{(7+2\theta)/8+a}.
\]
The displayed calculation uses level-one trace-formula notation;
Theorems~8.4--8.5 of \cite{LYTrace} also state results for
$\Gamma_0(N)$. We use this calculation to compare window exponents.
The full fixed-level and weight-uniform statement used below is
deduced separately from \cite{LLY}.

The stronger shifted convolution estimate in \cite{LLYJNT} leads,
in the near-conductor range considered there, to
\[
 \M_{K,L,Y}(W)
 \ll K^{1+\varepsilon}LY
       \left(1+\frac{K^2Y^{1/2+\theta}}{L^{4+2\theta}}\right).
\]
This is the leading term of \cite[(5.21)]{LLYJNT}; the remaining
expansion terms are no larger. At $Y\le K^2$, its bracket is
bounded for
\[
 L\ge K^{1-1/(4+2\theta)+a}.
\]
Here the coefficient moment is formulated in that paper for
$K^{2-\Delta}\le Y\le K^{2+\varepsilon}$, with a fixed small
$\Delta>0$. We do not infer an arbitrary-length coefficient
moment merely from its central-value theorem. The exponent
$1-1/(4+2\theta)=103/135$, at $\theta=7/64$, also occurs in
Ye's short-interval fourth moment \cite{Ye06}.

Thus the earlier window exponents relevant to \eqref{eq:5} are
$(7+2\theta)/8$ and $1-1/(4+2\theta)$, with the source qualifications
just stated. In \cite{LLY}, the moment is available throughout
$1\le Y\le K^2$ and the window can be shortened to
$L=K^{1/3+a}$ for arbitrarily small fixed $a>0$ sufficiently small.
This is the crucial improvement for our first theorem.

Indeed, if a moment of size $K^{1+\varepsilon}LY$ is available
with $L=K^{\gamma+a}$, positivity bounds a single coefficient sum by
\[
 (K^{1+\varepsilon}LY)^{1/2}
 =K^{(1+\gamma)/2+(a+\varepsilon)/2}Y^{1/2}.
\]
At $\gamma=1/3$, the exponent of $K$ before the weight loss is
$2/3$. At $Y=K^2$, the resulting size is $K^{5/3+o(1)}$.
Using only a window of length $K^{1/2+a}$ would give
$K^{7/4+o(1)}$ by the same argument. Lemma~\ref{lem:baire} below
allows us to quantify the additional loss from the nonlinear
weight by a finite derivative norm and absorb it into the
prescribed $K^\kappa$.

\section{A finite derivative norm controls the entire family}

The fixed-weight moment has an implied constant that may depend on every
derivative of the weight.  Before inserting the weights needed for the
nonlinear twist, we show that one finite collection of derivatives
controls the whole family of moments.  This is an application of the
Banach--Steinhaus principle.  We give the details, including the
completeness needed for the category argument; see also
\cite[Theorems~2.2 and~2.6]{Rudin}.

Fix a closed compact interval \(I\Subset(1,2)\).  We work over
\(\mathbb C\), and put
\begin{equation}\label{eq:9}
\begin{aligned}
 E=C_I^\infty&=\{W\in C^\infty(\mathbb R):\supp W\subset I\},\\
 p_r(W)&=\max_{0\le m\le r}\norm{W^{(m)}}_\infty,
 \qquad r=0,1,2,\ldots .
\end{aligned}
\end{equation}
The interval \(I\) is fixed: all members of \(E\) have their support
inside this same interval.  In particular, we are not taking the union
of spaces whose supports approach the endpoints of \((1,2)\).

For each \(r\), the function \(p_r\) satisfies
\[
 p_r(W+V)\leq p_r(W)+p_r(V),\qquad
 p_r(cW)=|c|p_r(W),\qquad p_r(W)\leq p_{r+1}(W).
\]
Also, \(p_r(W)=0\) implies \(W=0\), since the zeroth derivative is
included.  As usual, we refer to the \(p_r\) as seminorms when using
them to specify the derivative topology.  A sequence converges in this
topology precisely when every derivative of each fixed order converges
uniformly.  The sets defined by finitely many inequalities
\(p_r(W)<\rho_r\) are convex neighborhoods of zero.

\subsection{A complete metric for the derivative topology}

Define
\[
 d_E(W,V)=\sum_{r=0}^{\infty}2^{-(r+1)}
                   \min\{1,p_r(W-V)\}.
\]
The series converges because its summands are at most \(2^{-(r+1)}\).
This is a translation-invariant metric.  Positivity and symmetry follow
from those of the seminorms, while the triangle inequality follows from
\[
 \min\{1,a+b\}\leq\min\{1,a\}+\min\{1,b\}
 \qquad(a,b\geq0).
\]

The metric gives exactly the derivative topology.  Indeed, for every
fixed \(r\),
\[
 2^{-(r+1)}\min\{1,p_r(W-V)\}\leq d_E(W,V).
\]
Thus convergence in the metric implies convergence in every \(p_r\).
Conversely, given \(\epsilon>0\), first choose \(R\) so that
\[
 \sum_{r>R}2^{-(r+1)}=2^{-(R+1)}<\epsilon/2.
\]
If \(p_R(W-V)<\epsilon/2\), the increasing property of the
seminorms bounds the part of the metric sum with \(r\leq R\) by
\(\epsilon/2\), and the remaining part is smaller than
\(\epsilon/2\).  Hence \(d_E(W,V)<\epsilon\).  The same two
observations show that a sequence is Cauchy for \(d_E\) if and only
if it is Cauchy for every \(p_r\).

We now verify completeness.  Let \((W_n)\) be a Cauchy sequence.
For each nonnegative integer \(m\), the sequence \(W_n^{(m)}\) is
uniformly Cauchy on \(\mathbb R\), and therefore has a uniform limit
\(F_m\).  This limit is continuous and vanishes outside \(I\).
For any \(x,y\in\mathbb R\), the fundamental theorem of calculus gives
\[
 W_n^{(m)}(y)-W_n^{(m)}(x)
       =\int_x^y W_n^{(m+1)}(u)\,du.
\]
Uniform convergence permits passage to the limit: the error in the
integral is at most
\(|y-x|\norm{W_n^{(m+1)}-F_{m+1}}_\infty\).  Consequently,
\[
 F_m(y)-F_m(x)=\int_x^y F_{m+1}(u)\,du.
\]
Since \(F_{m+1}\) is continuous, this identity shows that
\(F_m'=F_{m+1}\).  Applying it successively for all \(m\), we find
that \(W=F_0\) is smooth and \(W^{(m)}=F_m\).  The function \(W\)
vanishes outside the closed interval \(I\), so \(\supp W\subset I\)
and \(W\in E\).  Finally, uniform convergence of the finitely many
derivatives of order at most \(r\) gives
\(p_r(W_n-W)\to0\) for every \(r\).  Thus \(d_E\) is complete.

It follows that \(E\) is a Fr\'echet space: its topology is locally
convex, is induced by a metric, and is complete.

\subsection{Continuity of each normalized moment}

\begin{lemma}[Finite-derivative transfer]\label{lem:baire}
For fixed \(N,g,d,a,I\), there exist an integer \(R\geq1\) and a
constant \(C>0\) such that
\begin{equation}\label{eq:10}
\M_{K,L,Y}(W)\leq C K^{1+a}LY\,p_R(W)^2
\end{equation}
for every \(W\in E\) and every triple in \emph{\eqref{eq:6}}.
The integer \(R\) and the constant \(C\) do not depend on \(W\).
\end{lemma}

\begin{proof}
For an admissible parameter triple \(P=(K,L,Y)\), write
\begin{equation}\label{eq:11}
q_P(W)=
 \left(\frac{\M_{K,L,Y}(W)}{K^{1+a}LY}\right)^{1/2}.
\end{equation}
We first show that \(q_P\) is a continuous seminorm for each fixed
\(P\).

The cuspidal Laplace spectrum of \(\Gamma_0(N)\) has finite
multiplicities and only finitely many eigenvalues below any fixed
bound.  This statement concerns the cuspidal spectrum, rather than the
continuous Eisenstein spectrum.  Every form in the window
\(|k_j-K|\leq L\) has Laplace eigenvalue at most
\(1/4+(K+L)^2\).  Hence the set \(\B_N(K,L)\) is finite.  We count
chosen normalized basis representatives, not their arbitrary scalar
multiples.

The coordinate corresponding to a form \(f_j\in\B_N(K,L)\) is
\[
 A_{P,j}(W)=\frac{1}{(K^{1+a}LY)^{1/2}}
        \sum_{n\geq1}\lambda_{f_j}(n)\lambda_g(n)W(n/Y).
\]
Only finitely many integers occur, since \(n/Y\) must belong to
\(I\Subset(1,2)\).  Their coefficients do not depend on \(W\).
It follows directly that
\[
 A_{P,j}(cW+bV)=cA_{P,j}(W)+bA_{P,j}(V)
 \qquad(c,b\in\mathbb C).
\]
Moreover,
\[
 |A_{P,j}(W)|\leq c_{P,j}p_0(W),\qquad
 c_{P,j}=\frac{1}{(K^{1+a}LY)^{1/2}}
       \sum_{n/Y\in I}|\lambda_{f_j}(n)\lambda_g(n)|<\infty.
\]
Thus every coordinate is continuous.

Let \(A_P(W)\) be the finite vector consisting of these coordinates.
Then \(q_P(W)=\norm{A_P(W)}_2\).  For finite complex vectors
\(u,v\), Cauchy--Schwarz gives
\[
 \norm{u+v}_2^2
 =\norm{u}_2^2+2\operatorname{Re}\langle u,v\rangle
                         +\norm{v}_2^2
 \leq(\norm{u}_2+\norm{v}_2)^2.
\]
Consequently,
\[
 q_P(W+V)\leq q_P(W)+q_P(V),\qquad
 q_P(cW)=|c|q_P(W).
\]
The coordinate estimates also give
\[
 q_P(W)\leq
   \left(\sum_{f_j\in\B_N(K,L)}c_{P,j}^2\right)^{1/2}p_0(W).
\]
The coefficient on the right is finite for this fixed triple.
In particular,
\(|q_P(W)-q_P(V)|\leq q_P(W-V)\) proves continuity of \(q_P\).

Proposition~\ref{prop:input} now says that
\begin{equation}\label{eq:12}
\sup_P q_P(W)<\infty\qquad\text{for each fixed }W\in E.
\end{equation}
If the source estimate is phrased only for sufficiently large \(K\),
with the threshold depending on \(W\), the same conclusion holds.
On a bounded interval \(2\leq K\leq K_0\), all admissible \(L\)
and \(Y\) are bounded.  Thus the integers and the basis forms that
can occur belong to finite sets.  The normalization denominator is
bounded away from zero, since \(K\geq2\), \(L\geq K^{1/3+d}\),
and \(Y\geq1\).  The preceding coordinate estimate therefore gives
a bound by a constant times \(p_0(W)\), uniformly on this bounded
parameter range.  Combining this with the large-\(K\) estimate
proves (12).

\medskip\noindent\textbf{The category argument.}

For integers \(m\geq1\), define
\begin{equation}\label{eq:13}
E_m=\{W\in E:q_P(W)\leq m\text{ for every admissible }P\}.
\end{equation}
For each fixed \(P\), its sublevel set \(\{W:q_P(W)\leq m\}\)
is closed by continuity.  An arbitrary intersection of closed sets
is closed, so every \(E_m\) is closed.  Equation \eqref{eq:12} implies that
\(E=\bigcup_{m\geq1}E_m\).

We recall the form of the Baire category theorem needed here: a
nonempty complete metric space cannot be a countable union of closed
sets all having empty interior.  Its proof in the present setting can
be given by nested balls.  If every \(E_m\) had empty interior, its
complement would be open and dense.  Choose a closed metric ball
\(B_1\) of positive radius, at most \(2^{-1}\), inside
\(E\setminus E_1\).  Given \(B_{m-1}\), the open ball with the same
center and radius meets \(E\setminus E_m\).  Within this open
intersection choose a closed ball \(B_m\) of positive radius at most
\(2^{-m}\).  Then \(B_m\subset B_{m-1}\) and
\(B_m\cap E_m=\varnothing\).

The centers of these balls form a Cauchy sequence: all centers after
the \(m\)-th lie in \(B_m\), whose radius tends to zero.
Completeness gives a limit.  For each \(m\), the tail of the sequence
lies in the closed set \(B_m\), so the limit also lies in \(B_m\).
It therefore belongs to none of the \(E_m\), contradicting their
cover of \(E\).  This proves that some \(E_m\) has nonempty
interior.

Choose an interior point \(W_0\in E_m\) and a number
\(0<\epsilon_0\leq1\) such that the metric ball
\(d_E(W_0,W)<\epsilon_0\) is contained in \(E_m\).  Choose
\(R\geq1\) with \(2^{-(R+1)}<\epsilon_0/2\), and set
\(r=\epsilon_0/2\).  If \(p_R(V)<r\), then the increasing
seminorms bound the first \(R+1\) summands in \(d_E(V,0)\) by
\(r\), and the remaining sum is smaller than \(\epsilon_0/2\).
Translation invariance thus gives
\(d_E(W_0+V,W_0)<\epsilon_0\).  We have obtained
\begin{equation}\label{eq:14}
p_R(V)<r\quad\Longrightarrow\quad W_0+V\in E_m.
\end{equation}
This explicitly converts the metric neighborhood into a condition
on one finite derivative norm.

Both \(W_0+V\) and \(W_0\) belong to \(E_m\).  Subadditivity and
homogeneity therefore imply, for every admissible \(P\),
\begin{equation}\label{eq:15}
q_P(V)=q_P\bigl((W_0+V)-W_0\bigr)
       \leq q_P(W_0+V)+q_P(W_0)\leq2m
       \qquad\bigl(p_R(V)<r\bigr).
\end{equation}
For \(W\ne0\), we have \(p_R(W)>0\).  Apply \eqref{eq:15} to
\[
 V=\frac{r}{2p_R(W)}W,
 \qquad p_R(V)=r/2<r.
\]
Using homogeneity once more gives
\begin{equation}\label{eq:16}
q_P(W)=\frac{2p_R(W)}r q_P(V)
        \leq\frac{4m}r p_R(W).
\end{equation}
For \(W=0\), this inequality is immediate.  Squaring \eqref{eq:16} and
using \eqref{eq:11} proves \eqref{eq:10}, with \(C=(4m/r)^2\).

The set of parameter triples \(P\) need not be countable.
The countable collection to which the Baire theorem applies is the
collection \(E_1,E_2,\ldots\).
\end{proof}

Lemma~\ref{lem:baire} supplies a single finite derivative norm before
the oscillating weights are introduced.  Its order depends only on
the fixed moment parameters and the common support interval.
Consequently, the resulting admissible range of \(\beta\) can be
chosen independently of \(\alpha\) and \(\phi\).

\section{The original cutoff and the individual estimate}

We give the partition explicitly, because \(\phi\) may be nonzero
near zero. Put \(q=3/2\), and choose
\(\chi\in C_c^\infty((1,2))\), \(\chi\geq0\), positive on an interval
\([u_0,u_1]\) with \(u_1/u_0>q\). For \(u>0\), define
\begin{equation}\label{eq:17}
D(u)=\sum_{\ell\in\mathbb Z}\chi(u/q^\ell),\qquad
 \psi(u)=\frac{\chi(u)}{D(u)}.
\end{equation}
The sum defining \(D\) is locally finite and positive, and
\(D(qu)=D(u)\). Consequently,
\begin{equation}\label{eq:18}
\sum_{\ell\in\mathbb Z}\psi(u/q^\ell)=1\qquad(u>0).
\end{equation}
Choose $I\Subset(1,2)$ once and for all to contain $\supp\psi$,
and apply Lemma~\ref{lem:baire} with this interval. Both $\psi$
and $I$ are fixed independently of $\alpha,\beta,\phi,K,L,X,Y$.

For an integer \(n\geq2\), a term \(\psi(n/q^\ell)\ne0\) requires
\(q^\ell>n/2\geq1\). Thus only scales \(Y=q^\ell\geq1\) are needed for
\(n\geq2\). The integer \(n=1\) will be kept separately.
For a scale \(Y=q^\ell\), let
\begin{equation}\label{eq:19}
W_{Y,X,\alpha,\beta}(u)=
       \psi(u)\phi(Yu/X)\e(\alpha Y^\beta u^\beta).
\end{equation}
All these functions belong to the same space \(E\). A nonzero
weight requires \(Y<X/\inf I<X\). Thus every contributing scale
satisfies \(1\leq Y\leq K^2\) when \(1\leq X\leq K^2\).

For every fixed integer \(R\), the product rule and repeated
differentiation of the exponential give
\begin{equation}\label{eq:20}
p_R(W_{Y,X,\alpha,\beta})
 \leq C_{R,I,\phi}(1+|\alpha|Y^\beta)^R,
 \qquad 0<\beta\leq1.
\end{equation}
To check the uniformity in \(\beta\), note that every derivative of
\(u^\beta\) of order at most \(R\) is uniformly bounded on \(I\) for
\(0<\beta\leq1\). Derivatives of \(\phi(Yu/X)\) contribute
\((Y/X)^m\leq1\). Each derivative of the exponential contributes a
polynomial in \(\alpha Y^\beta\) of degree at most the order of
differentiation. These observations prove \eqref{eq:20}, including uniform
constants as \(\beta\) approaches zero.

Define the hard-window moment for the original sum by
\begin{equation}\label{eq:21}
\M_{K,L}(X;\alpha,\beta)
 =\sum_{f_j\in\B_N(K,L)}
           |S_X(f_j,g;\alpha,\beta)|^2.
\end{equation}
For all \(n\geq2\), \eqref{eq:18} gives the exact decomposition
\begin{equation}\label{eq:22}
S_X(f_j,g;\alpha,\beta)
 =\e(\alpha)\phi(1/X)
   +\sum_{\substack{Y=q^\ell\\\ell\geq0}}
       T_{j,Y}(W_{Y,X,\alpha,\beta}).
\end{equation}
Only finitely many scales contribute for a fixed \(X\).

The fixed-level Weyl law (see \cite{I}) gives
\begin{equation}\label{eq:23}
\#\B_N(K,L)\ll_N KL+K\log K\ll_{N,d} KL
\end{equation}
in the range (6). Hence the first term in \eqref{eq:22} has spectral
Euclidean norm \(O_\phi((KL)^{1/2})\).

Apply Minkowski's inequality to \eqref{eq:22}, followed by \eqref{eq:10} and (20).
We get
\begin{equation}\label{eq:24}
\begin{aligned}
 \M_{K,L}(X;\alpha,\beta)^{1/2}
 &\ll_\phi (KL)^{1/2}
   +(K^{1+a}L)^{1/2}
       \sum_{Y<X}Y^{1/2}(1+|\alpha|Y^\beta)^R\\
 &\ll (K^{1+a}LX)^{1/2}(1+|\alpha|X^\beta)^R .
 \end{aligned}
\end{equation}
In the last step we used the geometric sum
\(\sum_{q^\ell<X}q^{\ell/2}\ll_q X^{1/2}\).
There is no logarithmic loss. We have proved the following
weight-uniform moment.
\begin{proposition}\label{prop:twisted}
For fixed \(0<d<1/3\) and \(a>0\), there is an integer
\(R=R(N,g,d,a)\geq1\), independent of \(\alpha,\beta,\phi\), such that
\begin{equation}\label{eq:25}
\M_{K,L}(X;\alpha,\beta)
 \ll_{N,g,d,a,\phi}
 K^{1+a}LX(1+|\alpha|X^\beta)^{2R}
\end{equation}
for the window in \eqref{eq:6}, \(1\leq X\leq K^2\), and \(0<\beta\leq1\).
\end{proposition}

Taking \(K=k_f\) and \(L=K^{1/3+d}\), the form \(f\) occurs in the
window, and positivity gives
\begin{equation}\label{eq:26}
\begin{aligned}
 |S_X(f,g;\alpha,\beta)|
 &\ll (K^{1+a}LX)^{1/2}(1+|\alpha|X^\beta)^R\\
 &=K^{2/3+(d+a)/2}X^{1/2}(1+|\alpha|X^\beta)^R .
 \end{aligned}
\end{equation}

\begin{proof}[Proof of Theorem \ref{thm:main}]
It suffices to consider \(0<\kappa<1/3\); a larger \(\kappa\) follows
from a smaller one. Choose \(d=a=\kappa/2\) in \eqref{eq:26}, and let \(R\)
be the corresponding integer. Put
\begin{equation}\label{eq:27}
\beta_0=\min\!\left(1,\frac{\kappa}{4R}\right).
\end{equation}
For \(1\leq X\leq K^2\) and \(0<\beta\leq\beta_0\),
\begin{equation}\label{eq:28}
(1+|\alpha|X^\beta)^R
 \leq(1+|\alpha|)^R K^{2R\beta}
 \leq(1+|\alpha|)^R K^{\kappa/2}.
\end{equation}
The initial exponent in \eqref{eq:26} is \(2/3+\kappa/2\).
Combining \eqref{eq:26} and \eqref{eq:28} proves (2). All derivative estimates used
above are uniform for \(0<\beta\leq1\), proving the claimed uniformity.
\end{proof}

As a concrete consequence with a nonnumerical range of $\beta$, take \(d=a=1/12\), \(L=K^{5/12}\), and \(X=K^2\). Equation \eqref{eq:26}
then reads
\begin{equation}\label{eq:29}
|S_{K^2}|\ll_{\alpha,\phi,N,g}K^{7/4+2R\beta}.
\end{equation}
Thus
\begin{equation}\label{eq:30}
0<\beta\leq\frac1{16R}
 \quad\Longrightarrow\quad
 |S_{K^2}|\ll K^{15/8}=X^{15/16}.
\end{equation}
Here \(R\) is chosen before \(\alpha,\beta,\phi\).

\begin{remark}
The proof does not compute $R$. Thus \eqref{eq:27} and \eqref{eq:30} assert the
existence of a positive range, but do not provide a numerical
endpoint. The next section gives an explicit range by retaining
the phase in the transformed sums.
\end{remark}

\section{A quantitative estimate with a longer window}

In this section the level is again any fixed \(N\). We use the
shifted-series bounds recorded in \cite[(1.7)--(1.10)]{LLYJNT}.
Write \(\theta=7/64\). The cubic bound alone gives the exponent
in \eqref{eq:54} below. The improvement from cubic to nearly linear growth
in the vertical parameter gives the stronger explicit result in
Theorem \ref{thm:numeric}.

\begin{theorem}\label{thm:numeric}
Fix \(0<\beta<5/32\), real \(\alpha\), and
\(\phi\in C_c^\infty(-1,1)\). For every \(\varepsilon>0\),
\begin{equation}\label{eq:31}
|S_{K^2}(f,g;\alpha,\beta)|
 \ll_{N,g,\alpha,\beta,\phi,\varepsilon}
 K^{423/224+(5/7)\beta+\varepsilon}.
\end{equation}
In particular, the exponent is below \(2\) if \(\varepsilon\) is
sufficiently small relative to \(5/32-\beta\).
\end{theorem}

We prove a moment estimate first. For a fixed smooth function
\(A\) supported in a fixed compact interval in \((1,2)\), put
\begin{equation}\label{eq:32}
U_{j,Y}=\sum_{n\geq1}\lambda_{f_j}(n)\lambda_g(n)
                A(n/Y)\e(\alpha n^\beta),\qquad
 F=|\alpha|Y^\beta,\quad Z=1+F+K/L.
\end{equation}
For fixed \(\delta,b>0\), assume
\begin{equation}\label{eq:33}
1\leq Y\leq K^2,\qquad
 K^{1/2+\delta}\leq L\leq K^{1-\delta},\qquad
 F\leq KL K^{-b}.
\end{equation}
The constants below can depend on a fixed finite number of
seminorms of \(A\).

\begin{proposition}\label{prop:numericmoment}
Under \eqref{eq:33}, for every \(\varepsilon>0\),
\begin{equation}\label{eq:34}
\sum_{f_j\in\B_N(K,L)}|U_{j,Y}|^2
 \ll K^\varepsilon
       \left\{KLY(1+F)+Y^{3/2+\theta}Z^{5/2}\right\}.
\end{equation}
If only the cubic shifted-series bound is used, the same argument
proves \eqref{eq:34} with \(Z^{9/2}\) in place of \(Z^{5/2}\).
\end{proposition}

\subsection{Kuznetsov, the kernel expansion, and Voronoi}

For clarity, all truncation powers in this proof are initially
\(K^a\), where \(a>0\) is chosen sufficiently small in terms of
\(\delta,b,\varepsilon\). Finite occurrences of these powers are
absorbed in the final \(K^\varepsilon\).
Apply Kuznetsov with the positive even spectral test
\begin{equation}\label{eq:35}
H_{K,L}(t)=\frac{t^2+1/4}{K^2+1/4}
 \left\{\frac1{\cosh((t-K)/L)}
                 +\frac1{\cosh((t+K)/L)}\right\}.
\end{equation}
It is bounded below by a positive constant on the hard window.
It is nonnegative on the real axis and at exceptional spectral
parameters $t\in i[-1/2,1/2]$. The factor $t^2+1/4$ makes
$H_{K,L}(\pm i/2)=0$, as required in a usual version of Kuznetsov.
After $t=K+Lu$, this correction merely replaces $1/\cosh u$ by
a linear combination of $1/\cosh u$, $u/\cosh u$, and
$u^2/\cosh u$, with uniformly bounded coefficients. All kernel
estimates used below remain uniform. Use Fourier coefficients
on the trace formula side. Apply \eqref{eq:8} only to the primitive
hard-window contribution, where \(k_j\asymp K\); this costs only
an arbitrarily small power of \(K\). The full trace formula keeps
its harmonic normalization and is used by positivity.
The continuous contribution
is nonnegative. The original diagonal is
\(O(K^\varepsilon KLY)\), by the fixed-\(g\) Rankin--Selberg
bound \(\sum_{n\leq u}|\lambda_g(n)|^2\ll_g u\).

The off-diagonal kernel restricts the modulus to
\begin{equation}\label{eq:36}
N\mid c,\qquad c\leq K^a C,\qquad C=Y/(KL).
\end{equation}
Use the quadratic kernel expansion in
\cite[Sections 4.1--4.4]{LY}; the wider range \eqref{eq:33} is precisely
where that expansion is valid. The more flexible expansion of
\cite[Lemma 5.1]{LLY} is available for shorter windows, but the
indices in \eqref{eq:52} below refer to the quadratic version. In applying
the small-argument estimate one must also control the infinite
modulus tail. Shift the spectral contour in the Kuznetsov Bessel
integral to \(\operatorname{Im}t=-3/4\). This gives
\(V_{K,L}(x)\ll_A K^{-A}x\) for \(0<x\leq1\).
The possible pole at $t=-i/2$ is canceled by the zero of
$H_{K,L}$. On the shifted contour the real part of the Bessel
order is $3/2$. Its power series and repeated integration by parts
in the localized spectral variable give
$O_A(K^{-A}x^{3/2})$, which implies the stated bound. Together with the Weil bound,
this makes the infinite \(c\)-tail absolutely summable. For
\(x\geq1\) in the excluded range, the rapid-decay estimate for the
kernel is summed over only finitely many moduli.

Open the Kloosterman sum and apply Voronoi to the second
coefficient of \(g\), following \cite[Sections 4.5--4.8]{LY}.
There is no change of cusp here: since \(N\mid c\), the rational
point in question comes from a matrix in \(\Gamma_0(N)\).
The transformed form is therefore the same fixed form \(g\).
The Maass \(K\)-Bessel branch is negligible: its argument is at
least a constant times
\(\sqrt Y/c\geq KLK^{-a}/\sqrt Y\geq LK^{-a}\).

Writing the dual index as \(r=n+h\), a typical oscillatory phase,
in the variables \(n=Yz\), \(y=Yw^2/z\), is
\begin{equation}\label{eq:37}
\begin{split}
 \Phi(z,w)={}&\frac{2Yw}{c}
       \left(\eta_1+\eta_2\sqrt{1+\frac h{Yz}}\right)
       -\frac{\eta_3K^2c}{4\pi^2Yw}\\
       &+\alpha Y^\beta
             \left(z^\beta-w^{2\beta}z^{-\beta}\right),
       \qquad \eta_1,\eta_2,\eta_3\in\{\pm1\}.
 \end{split}
\end{equation}
It is enough to estimate all the displayed sign choices
uniformly, including any combinations absent from the expansion. The compact support bounds \(z,w\)
above and below by positive constants.
Because \(F\ll (Y/c)K^{a-b}\), the branch
\(\eta_1=\eta_2\) is nonstationary. The opposite-sign branch
first restricts \(r\asymp Y\), by the first derivative test;
the power separation in \eqref{eq:33} makes this uniform. It then
restricts
\begin{equation}\label{eq:38}
|h|\leq K^a cQ_c,\qquad
 Q_c=1+F+K^2c/Y.
\end{equation}
Indeed, the three variable phase scales in the \(w\)-derivative
are \(|h|/c\), \(F\), and \(K^2c/Y\). If \(|h|/c\) exceeds
the other two by a sufficiently large power of \(K\), repeated
integration by parts is arbitrarily saving. Moreover,
\begin{equation}\label{eq:39}
\frac{K^a cQ_c}{Y}
 \ll K^{2a-b}+K^{3a}/L^2+K^{2a}/(KL)=o(1).
\end{equation}
Thus all surviving dual indices satisfy \(r=n+h\asymp Y\),
and all normalized amplitudes below have uniform compact support.
The extra \(cF\) in \eqref{eq:38} is retained throughout.

Before taking absolute values, the leading transformed term has
exactly the arithmetic factor
\begin{equation}\label{eq:40}
L\sqrt Y\sum_{\substack{c\leq K^a C\\N\mid c}}
         \frac1c\sum_h R_c(h)P(c,h,Y),
 \qquad
 R_c(h)=\sum_{d\bmod c}^{*}\e(dh/c).
\end{equation}
Here \(P\) is the shifted sum in \cite[(4.11)]{LY}, with the
nonlinear phase inserted. This factor follows directly from
\cite[(4.10)]{LY}: its modulus power is \(c^{-1}\) for the
leading term. The extra power of \(c\) in that paper's (4.11)
comes only from the subsequent bound \(|R_c(h)|\leq c\).
We retain (40).

We specify the weights for the leading term, in order to keep track
of the scale in its Mellin transform. We use the Fourier convention
$\widehat v(\xi)=\int_{\mathbb R}v(u)e(u\xi)\,du$. Put
\[
 \mathcal H(\xi)=\widehat q(\xi)+(L/K)\widehat{uq}(\xi),
 \qquad
 q(u)=\frac{(K+Lu)^2+1/4}{K^2+1/4}\frac1{\cosh u}.
\]
The two packets in \eqref{eq:35} are
$q((t-K)/L)$ and $q(-(t+K)/L)$; thus the second packet is
the reflection of the first, even though $q$ need not be even.
The functions $q$ and $\mathcal H$ have uniformly bounded Schwartz
seminorms, because $L/K\le1$. Up to a fixed coefficient and a
choice of signs, the leading shifted sum is
\[
 \begin{split}
 P(c,h,Y)=\sum_{\substack{n\ge1\\n+h\ge1}}
 &\frac{\lambda_g(n)\lambda_g(n+h)}{(n(n+h))^{1/4}}
 A(n/Y)e(\alpha n^\beta)B(n,n+h),\\
 B(n,r)=K\int_0^\infty
 &\overline{A(v)}v^{-1/2}
 \mathcal H\left(\frac{\eta_3LKc}{2\pi^2\sqrt{vYn}}\right)\\
 &\times e\left(\frac{2\sqrt{vY}}c
              (\eta_1\sqrt n+\eta_2\sqrt r)
 -\frac{\eta_3K^2c}{4\pi^2\sqrt{vYn}}
 -\alpha Y^\beta v^\beta\right)\,dv.
 \end{split}
\]
These are the leading weights in \cite[(4.10)]{LY} after the
nonlinear twist has been inserted. The harmless changes of sign
in the Fourier convention can be included among the sign choices.

Write $d=h/Y$ and
\[
 \rho_g(z,d)=
 \left(\frac{\sqrt{z(z+d)}}{2z+d}\right)^{2ik_g}.
\]
This factor and its reciprocal have bounded derivatives on our
support, since $g$ is fixed and $|d|=o(1)$. For the leading term
define
\[
 \begin{split}
 G_0(z)={}&2K\rho_g(z,d)^{-1}z^{-3/4}(z+d)^{-1/4}A(z)\\
 &\times\int_0^\infty
 \overline{A(w^2/z)}
 \mathcal H\left(\frac{\eta_3LKc}{2\pi^2Yw}\right)
 e\bigl(\Phi(z,w)\bigr)\,dw.
 \end{split}
\]
The substitution $v=w^2/z$ gives precisely the phase in (37).
If $x=n+(n+h)=2Yz+h$, let $G(x)=Y^{-1/2}G_0(z)$.
Then
\[
 \widehat G(s)=\int_0^\infty G(x)x^{s-1}\,dx
 =(2Y)^sY^{-1/2}\int_0^\infty
 G_0(z)(z+d/2)^{s-1}\,dz.
\]
The inverse factor $\rho_g^{-1}$ is exactly what removes the ratio
factor in the shifted Dirichlet series. This identity also shows
the factor $Y^{\sigma-1/2}$ without suppressing a power of $Y$.

\subsection{The Mellin transform and a Plancherel estimate}

For $h\ne0$, initially in the half-plane of absolute convergence,
put
\[
 D_g^+(s,1,1,h)=\sum_{\substack{n,r\ge1\\r-n=h}}
 \lambda_g(n)\lambda_g(r)
 \left(\frac{\sqrt{nr}}{n+r}\right)^{2ik_g}(n+r)^{-s}.
\]
Mellin inversion, as in \cite[(4.13)--(4.15)]{LY},
expresses the leading shifted sum as
\begin{equation}\label{eq:41}
P(c,h,Y)=\frac1{2\pi i}\int_{(\sigma)}
                 D_g^+(s,1,1,h)\widehat G(s)\,ds,
 \qquad \sigma=\tfrac12+\theta+\varepsilon_1.
\end{equation}
The superscript \(+\) means positive indices only, with the
fixed ratio factor required for Maass \(g\). Its inverse is
incorporated into \(G\). That ratio and all its derivatives are
bounded on the support \(n,n+h\asymp Y\).

The published bound for the all-signed Maass series is sufficient.
Negative-negative pairs duplicate the positive-positive pairs
because the coefficients are real and have a fixed parity.
The remaining mixed-sign pairs have
\(|m|+|n|=|h|\). Each such term in the inverse Mellin integral
is multiplied by \(G(|h|)=0\), since \(|h|=o(Y)\) and
\(G\) is supported at scale \(Y\). Thus \eqref{eq:41} can be evaluated
with one half of the all-signed series, without a mixed-sign
error. This reasoning may first be made on an absolutely
convergent line and then continued; \(\sigma\) in \eqref{eq:41} lies
to the right of all exceptional poles. Negative \(h\) is handled
by exchanging the two indices.

From \cite[(1.7)--(1.10)]{LLYJNT}, choose an arbitrarily small
\(0<\varepsilon_1<1/4\). We may use the global estimate
\begin{equation}\label{eq:42}
|D_g(\sigma+it,1,1,h)|
 \ll |h|^{O(\varepsilon_1)}
       \left\{(1+|t|)^{1+\varepsilon_1}
                      +|h|^{1/2-\theta}\right\}.
\end{equation}
For \(|t|\geq1\) this is their improved bound; for
\(|t|\leq1\) their earlier cubic bound gives the same assertion.
The latter step is important, since the improved bound is stated
only for \(|t|\geq1\). For holomorphic \(g\), the second term
in braces can be omitted.

In the variable change leading to \cite[(4.17)]{LY}, both the
spectral correction \(K^2c/(4\pi^2Yw)\) and the Fourier-transform
amplitude become independent of \(z\). Consequently,
\begin{equation}\label{eq:43}
\norm{G_0^{(m)}}_\infty
 \ll_m K(1+F+|h|/c)^m.
\end{equation}
The \(K\) accounts for the leading spectral amplitude;
\(L\leq K\). The support is fixed. More explicitly, set $u=\log(z+d/2)$ and
\[
 V(u)=(2Y)^\sigma Y^{-1/2}e^{\sigma u}
                     G_0(e^u-d/2).
\]
Then $\widehat G(\sigma+it)=e^{it\log(2Y)}
\int V(u)e^{itu}\,du$. The support of $V$ stays in a fixed
compact interval, and
\begin{equation}\label{eq:44}
\norm V_2\ll K Y^{\theta+\varepsilon_1},\qquad
 \norm{V^{(m)}}_2
   \ll_m K Y^{\theta+\varepsilon_1}(K^aQ_c)^m.
\end{equation}
The factor \(Y^{\theta+\varepsilon_1}\) comes from
\((2Y)^sY^{-1/2}\), exactly as in (4.13) of \cite{LY}.

We spell out the transform estimate. If
\(\norm{V^{(m)}}_2\ll_m A Q^m\), \(Q\geq1\), then for fixed
\(p\geq0\) and any integer \(m>p+1/2\),
\begin{equation}\label{eq:45}
\int_{\mathbb R}(1+|t|)^p|\widehat V(t)|\,dt
 \ll_p A Q^{p+1/2}.
\end{equation}
For \(|t|\leq Q\), Cauchy--Schwarz and Plancherel give
\(\ll A Q^{p+1/2}\). For \(|t|>Q\), apply
Cauchy--Schwarz with weights \(|t|^{p-m}\) and \(|t|^m\).
The first squared integral is \(O(Q^{2p-2m+1})\), and the
second square root is \(O(AQ^m)\), giving the same bound.
This proves (45). Apply it with \(p=1+\varepsilon_1\), \(m=2\),
and also with \(p=0\), \(m=1\).
Equations \eqref{eq:41}--\eqref{eq:45} give
\begin{equation}\label{eq:46}
|P(c,h,Y)|
 \ll K^{1+O(a+\varepsilon_1)}Y^{\theta+\varepsilon_1}
    \left\{Q_c^{3/2+\varepsilon_1}
                  +|h|^{1/2-\theta}Q_c^{1/2}\right\}.
\end{equation}
Since \(Y\leq K^2\), we have \(c\leq Q_c\). From \eqref{eq:38},
\begin{equation}\label{eq:47}
|h|^{1/2-\theta}Q_c^{1/2}
 \ll K^{O(a)}Q_c^{3/2-2\theta}\leq K^{O(a)}Q_c^{3/2}.
\end{equation}
The additional Maass term is therefore absorbed.

\subsection{Ramanujan averaging, the zero shift, and remainders}

For every real \(H\geq0\),
\begin{equation}\label{eq:48}
\sum_{0<|h|\leq H}|R_c(h)|
 \leq\sum_{0<|h|\leq H}(c,h)
 =2\sum_{d\mid c}\varphi(d)\lfloor H/d\rfloor
 \leq2H\tau(c).
\end{equation}
This bound has no \(+c\) error; the zero shift is excluded.
Use \eqref{eq:46}--\eqref{eq:48} in \eqref{eq:40}, with \(H=K^acQ_c\). Since
\(Q_c\ll K^a Z\) and the number of moduli is at most \(K^aC\),
\begin{equation}\label{eq:49}
\begin{split}
 \M_{\ne0}
 &\ll K^{1+O(a+\varepsilon_1)}L Y^{1/2+\theta+\varepsilon_1}
             \sum_{c\leq K^aC}Q_c^{5/2+\varepsilon_1}\tau(c)\\
 &\ll K^\varepsilon Y^{3/2+\theta}Z^{5/2}.
 \end{split}
\end{equation}
All powers absorbed in \(K^\varepsilon\) are legitimate because
\(Y\leq K^2\), \(F\ll KL\), and \(c,Q_c,Z\) are bounded by
fixed powers of \(K\). If \(K^a C<1\), the modulus sum is empty.

For \(h=0\), the two square-root phases cancel in the opposite-sign
branch. Its \(w\)-derivative has the two scales
\(K^2c/Y\) and \(F\). It is nonstationary unless
\begin{equation}\label{eq:50}
c\ll K^a(1+F)Y/K^2.
\end{equation}
Derivatives of the Fourier amplitude cost
\(1+LKc/Y=1+(L/K)(K^2c/Y)\). Integration by parts outside \eqref{eq:50}
therefore gains both a power from \(K^2c/Y\) and the small ratio
\(L/K\). It can be repeated sufficiently many times to make
that contribution negligible. On each remaining modulus,
\(|B(n,n)|\ll K\), and Rankin--Selberg gives
\(|P(c,0,Y)|\ll K\sqrt Y\). Since
\(|R_c(0)|/c=\varphi(c)/c\leq1\), its contribution in \eqref{eq:40} is
\(O(KLY)\). There are at most \(O(K^a(1+F))\) such moduli,
so
\begin{equation}\label{eq:51}
\M_0\ll K^\varepsilon KLY(1+F).
\end{equation}
This also absorbs the original diagonal.

We record the control of expansion terms, rather than restricting
the proof to the leading one. In the expansion of \cite{LY}, a
term with indices $(\mu,\nu,j,k)$, $2\mu\le\nu$, put
$\lambda=k-\mu$. Before estimating the shifted sum, its front
factor and arithmetic sum are
\[
 L^{1-2\nu+2k}Y^{(1-\lambda-j)/2}
 \sum_c c^{\lambda+j-1}\sum_h R_c(h)P_{\mu,\nu,j,k}(c,h,Y).
\]
The shifted weight contains
$n^{-\lambda/2-1/4}(n+h)^{-j/2-1/4}$ in place of
$(n(n+h))^{-1/4}$. Its Mellin transform has the scale
$K Y^{\theta+\varepsilon_1-(\lambda+j)/2}$.
Thus the ratio of its bound to the leading bound, before summing
over $c$, is
\[
 L^{-2\nu}(Y/c)^\mu(c/Y)^j(L^2c/Y)^k.
\]
Since $\nu\ge2\mu$, this is at most
\begin{equation}\label{eq:52}
\left(\frac{Y}{cL^4}\right)^\mu
 \left(\frac cY\right)^j
 \left(\frac{L^2c}{Y}\right)^k.
\end{equation}
The first factor is at most \((K^2/L^4)^\mu\), and the last
is at most \((K^aL/K)^k\). They and the middle factor are
bounded in \eqref{eq:33}, for sufficiently small \(a\). Their normalized
weights satisfy the same derivative bounds (43). The Ramanujan
factor and the saving in \eqref{eq:48} remain present in each term.
For the zero shift the corresponding bound is
$P_{\mu,\nu,j,k}(c,0,Y)\ll K Y^{(1-\lambda-j)/2}$,
so the same multiplier also controls \eqref{eq:51}.

The small parameters in the kernel expansion are
\(\sqrt Y/L^2\leq K^{-2\delta}\) and
\(L^2/x\ll K^aL/K\leq K^{a-\delta}\), where
\(x\asymp Y/c\). Taking sufficiently many terms controls the
kernel remainder. The fixed-\(g\) oscillatory Voronoi expansion
also has a uniformly large Bessel argument, at least \(LK^{-a}\);
its remainder is controlled by taking sufficiently many terms.
The absolute remainder estimates before transformation are
unaffected by the unit-modulus twist. In the integration-by-parts
truncations after transformation, derivatives of the twist have
scale \(F\), and the power separation in \eqref{eq:33} is the required
margin. These are the valid reductions of
\cite[Sections 4.1--4.13]{LY}; we do not use Sections 4.14--4.15
or the shorter claimed range \(|h|\ll cK^a\).
Combining \eqref{eq:49} and \eqref{eq:51}, including all expansion terms, proves
Proposition \ref{prop:numericmoment}.

\subsection{The individual exponent and the original cutoff}

Apply the partition \eqref{eq:17}--(22). Its normalized dyadic amplitudes
are \(A_Y(u)=\psi(u)\phi(Yu/X)\), which have uniformly bounded
smooth seminorms; the nonlinear phase is already explicit in
Proposition \ref{prop:numericmoment}. Minkowski and geometric
summation give
\begin{equation}\label{eq:53}
\M_{K,L}(X;\alpha,\beta)
 \ll K^\varepsilon
  \left\{KLX(1+|\alpha|X^\beta)
     +X^{3/2+\theta}(1+|\alpha|X^\beta+K/L)^{5/2}\right\}
\end{equation}
for \(1\leq X\leq K^2\) satisfying the corresponding condition
in (33). The separate \(n=1\) term is absorbed as before.

For comparison, retaining only the cubic shifted-series bound
in \eqref{eq:42}, and using \(p=3,m=4\) in \eqref{eq:45}, gives \(Z^{9/2}\).
At \(X=K^2\), balancing the two terms then yields
\begin{equation}\label{eq:54}
L=K^{151/176-(4/11)\beta},\qquad
 |S_{K^2}|\ll K^{679/352+(9/11)\beta+\varepsilon},
 \qquad 0<\beta<25/288.
\end{equation}
Thus the more elementary cubic bound already gives a numerical
range. To prove Theorem \ref{thm:numeric}, use \eqref{eq:53} and set
\begin{equation}\label{eq:55}
\eta=\frac{5+4\theta-4\beta}{7}
       =\frac{87}{112}-\frac{4\beta}{7},\qquad L=K^\eta.
\end{equation}
For \(0<\beta<(1-2\theta)/5=5/32\),
\(2\beta<1-\eta\). Fixed \(\alpha\) therefore gives
\(|\alpha|K^{2\beta}=o(K/L)\), with a power separation.
The exponent \(\eta\) lies between \(11/16\) and \(87/112\),
so \eqref{eq:33} holds for some fixed \(\delta,b>0\).
At \(X=K^2\), positivity and \eqref{eq:53} give the two individual
bounds with exponents
\begin{equation}\label{eq:56}
\frac{3+\eta}{2}+\beta,
 \qquad \frac{11}{4}+\theta-\frac{5\eta}{4}.
\end{equation}
They agree for \eqref{eq:55}, and their common value is
\begin{equation}\label{eq:57}
\frac{13+2\theta}{7}+\frac{5\beta}{7}
       =\frac{423}{224}+\frac{5\beta}{7}.
\end{equation}
It is below \(2\) exactly for \(\beta<5/32\).
This proves Theorem \ref{thm:numeric}.

\begin{remark}
The two theorems have different strengths. Theorem~\ref{thm:main}
approaches $K^{5/3}$ at $X=K^2$, but its range of $\beta$ depends
on the requested loss and is not numerical. Theorem~\ref{thm:numeric}
gives an explicit range, with a weaker exponent. If $F$ exceeds
$K/L$, the nonlinear phase can enlarge both the shift range and
the Mellin-frequency range. The moment in \eqref{eq:53} records both
effects; improving it in that range is a separate problem.
\end{remark}

\section*{Declaration of AI Use}
During the preparation of this work, the 
author used OpenAI's GPT-6 Astra to 
implement proof strategy, assist 
with computations, support literature 
searches, and typeset the manuscript. 
The author carefully reviewed, 
verified, and edited all AI-generated 
content as necessary and takes full 
responsibility for the accuracy, 
originality, and integrity of the 
published article.

\end{document}